\title{ A dynamic point of view on universality for random matrices over finite local rings}
\author{Nikita Lvov\footnote{nikita.lvov@mail.mcgill.ca}}
\documentclass{article}
\usepackage{marginnote}
\usepackage{relsize}

\usepackage{ifxetex,ifluatex}
\if\ifxetex T\else\ifluatex T\else F\fi\fi T%
  \usepackage{fontspec}
\else
  \usepackage[T1]{fontenc}

  \usepackage{blkarray, bigstrut}
  \usepackage[utf8]{inputenc}
  
  \usepackage{amsthm}
  \usepackage{thmtools}
  \usepackage{mathtools}

  \usepackage{lmodern}
  \usepackage{amssymb}
  \usepackage{amsfonts}
  \usepackage{tikz-cd}
  \usepackage{amsthm}
  \usepackage{xfrac}
  \usepackage{mathtools}
  \usepackage{multirow}
  \usepackage{mathrsfs}
  \usepackage{comment}

  \newtheorem*{corollary}{Corollary}

  \newcommand{\probP}{\text{I\kern-0.15em P}}
  \newcommand{\probE}{\text{I\kern-0.15em E}}
  \newcommand{\coker}{\text{coker}}
  \newcommand{\defeq}{\overset{\mathrm{def}}{=\joinrel=}}

  \numberwithin{equation}{section}

  \theoremstyle{remark}
  \newtheorem*{remark}{Remark}
  \theoremstyle{remark}

  \newtheorem*{definition}{Definition}

  \newcommand{\rows}{l}
  \newcommand{\cols}{k}

  \newcommand{\Z}{\mathbb{Z}}

  \excludecomment{mysection}
  \excludecomment{mymysection}
  \excludecomment{maybeinclude}

    \newenvironment{mat}
    {
    \left[
\begin{array}}
{ \end{array}
\right]
}

\newcommand{\probM}{\mathcal{M}}
\newcommand{\probU}{\mathcal{U}}

\usepackage{mdframed}

\theoremstyle{theorem}

\newtheorem{prop}{Proposition}

\newtheorem{theorem}{Theorem}

\newtheorem*{theorem*}{Theorem}
\newtheorem{lemma}{Lemma}[theorem]
\newenvironment{ftheo*}
  {\begin{mdframed}\begin{theorem*}}
  {\end{theorem*}\end{mdframed}}

\newtheorem*{goal}{Goal}

  \usepackage{enumitem}

   \DeclareSymbolFont{bbold}{U}{bbold}{m}{n}
   \DeclareSymbolFontAlphabet{\mathbbold}{bbold}
   \newcommand{\one}{\mathbbold{1}}
   
   \theoremstyle{remark}

\fi

 \theoremstyle{definition}

 \newcommand{\comm}[1]{} 

 \newcommand{\lone}{L^1(X_0)}

 \usepackage{cite}
 \newcommand{\commm}[1]{}
 
 \newcommand{\chapter}{section}

       \usepackage{
       hyperref,
       cleveref,
       }
\begin{document}
\maketitle

\abstract{We consider the cokernel corners process for an i.i.d. matrix with entries in a finite local ring. When the distribution of the entries is uniform, this process is a Markov chain, and hence the ergodic theorem for Markov chains can be applied. This implies, in particular, that for uniformly distributed $p$-adic random matrices, the cokernels of the corners are distributed according to the Cohen-Lenstra measure, almost surely. The purpose of this note is to show that the conclusion of the ergodic theorem also holds for i.i.d matrices, provided that the distribution of the entries is not concentrated on the translate of a subring, or the translate of an ideal. This will follow from the bounds proved in a previous paper of the author.}

\renewcommand{\rows}{n}
\renewcommand{\cols}{n+u}
\newcommand{\Mm}{M}
\newcommand{\Mmodule}{\mathbf{M}}
\newcommand{\projj}{\text{proj}}
\newcommand{\One}{\mathlarger{\mathlarger{\one}}}
\newcommand{\probMM}{
\mathcal{M}_{\rows,\cols}
}
\newcommand{\probMMM}{\mathcal{N}}

\newcommand{\linf}[1]{ \Big| \Big| #1 \Big| \Big|_{l^\infty} }

\renewcommand{\lone}[1]{ \Big| \Big| #1 \Big| \Big|_{l^1} }

\newcommand{\ltwo}[1]{
\Big| \Big| #1 \Big| \Big|_{l^2}
}

\newcommand{\largematrix}[1]{
\begin{mat}{ccc|ccc}
&&&*&\hdots&\\
&#1&&\vdots&&\\
&&&*&\hdots&\\ \hline
*&&*&*&\hdots&\\
\vdots&&\vdots&\vdots&\ddots&\\
&&&&&
\end{mat}
}

\renewcommand{\lone}[1]{
\Big| \Big| #1 \Big| \Big|_{l^1}
}

\newcommand{\maxmod}{
\mathfrak{m}
}

\newcommand{\maxbound}
{
\max
\left(
\linf{\xi \mod \maxmod}
\,
,
\,
\frac{1}{ char(R/\maxmod) }
\right)
}

\newcommand{\rev}[1]{\reversemarginpar\marginpar{#1}\reversemarginpar}

\newcommand{\infmatrix}[1]
{
\begin{mat}{ccc|ccc}
&&&*&\hdots&\\
&#1&&\vdots&&\\
&&&*&\hdots&\\ \hline
*&\hdots&*&*&\hdots&\\
\vdots&&\vdots&\vdots&\ddots&\\
&&&&&
\end{mat}
}


\section{Introduction}

First, let $\mathcal{U}_{n,m}$ be an $n \times m$ matrix over
$\Z_p$, whose
entries are sampled uniformly at random. A theorem of Friedman and
Washington describes the asymptotic distribution of
$coker(\mathcal{U}_{n,n})$:

\begin{theorem*} \cite[Proposition 1]{FriedmanWashington1989}
\label{thm:friedmanwashington}
\begin{equation}
\label{eqn:friedmanwashington}
\lim_{n \rightarrow \infty} \probP\Big( coker(\mathcal{U}_{n,n}) \cong A \Big)
=
\frac{c_0}{|Aut(A)|}
\end{equation}
where
\[
c_0=\prod_{i=1}^{\infty} \left( 1 - \frac{1}{p^{i}}  \right)
\]
\end{theorem*}

The distribution (\ref{eqn:friedmanwashington}) on $p$-groups is known
as the Cohen-Lenstra distribution \cite{CohenLenstra}. More generally, \begin{equation}
\label{eqn:friedmanwashingtonu}
\lim_{n \rightarrow \infty} \probP
\Big(
coker(\mathcal{U}_{n,n+u})
\cong
A
\Big)
=
\frac{c_u}{|A|^{u} |Aut(A)|}
\hspace{0.5in} \text{ for $u \geq 0$}
\end{equation}
where
\[
c_u = \prod_{i=u+1}^{\infty} \left( 1 - \frac{1}{p^i} \right)
\]
\newline
\newline
From the recent work of Sawin and Wood
\cite[Lemma 6.7 and Lemma 6.6]{SawinWood}, we can deduce a formula valid for any
\textit{finite} local ring $R$. We consider an $n \times (n+u)$ matrix over $R$,
whose entries are independent and uniformly distributed. We again
denote this matrix as $\probU_{n,n+u}$. Now, \cite[Lemma 6.7]{SawinWood} implies
that for $u>0$, and any finite local ring $R$,
\[
\lim_{n \rightarrow \infty} \probP(coker(\probU_{n,n+u})=A)
=
\]
\begin{equation}
\label{eqn:sawinwood}
=\frac{1}{|A|^u |Aut(A)|}
\prod^{\infty}_{i=d(A)+u+1} \left(1 - \frac{1}{q^i} \right)
\end{equation}
where $q$ is the cardinality the residue field of $R$.
$d(A)$ is defined to be the difference between the number of relations
and the number of elements in the minimal presentation of
$A$, negative if there are more relations than
elements\footnote{For example, $d(R^k)=k$.}. We denote the measure on the right hand side of (\ref{eqn:sawinwood}) as $\mu_u$\rev{$\mu_u$}\footnote{For the reader's convenience, whenever new notation is introduced, this will marked in the left margin.}.
\newline
\newline
The expressions (\ref{eqn:friedmanwashington}) and (\ref{eqn:friedmanwashingtonu}) can both be deduced from (\ref{eqn:sawinwood}). 

\paragraph{Ergodic averages in the uniform case.} 
We slightly refine the above set-up. $\mathcal{U}$ will denote an infinite random matrix over $\Z_p$, whose entries are sampled uniformly and independently at random. Denote by $\probU_{n,m}$ the top left $n \times m$ corner of $\probU$.
\newline
\newline
From \cite{markovchainsoverr}, it follows that the random groups $\coker(\probU_{n,n+u})$ form a recurrent Markov chain. In \cite{markovchainsoverr}, we denote the generator of this Markov chain as $\Delta_u$\rev{$\Delta_u$}; by (\ref{eqn:sawinwood}) , the stationary measure of this Markov chain is $\mu_u$. From the ergodic theorem for Markov chains, it follows that:
\begin{equation}
\label{eqn:ergodicuniform}
\frac{1}{N} \sum_{i=1}^N \One_{(coker(\probU_{i,i+u})=A)}
\xrightarrow{N \rightarrow \infty}
\mu_{u}(A) \text{ a.s.}
\end{equation}

\subsection{Random matrices with i.i.d. entries that are not necessarily uniformly distributed.}

Let $R$ be a finite local ring. Let $\probM$\rev{$\probM$} be an infinite random matrix over $R$ whose entries are i.i.d. random variables, subject to the condition that their distribution is not supported on the translate of a subring of $R$ or the translate of an ideal.
\newline
\newline
Let $\probM_{n,m}$ \rev{$\probM_{n,m}$} denote the top left $n \times m$ corner of $\probM$. The next theorem states that the asymptotic distribution of $coker(\probM_{n,n+u})$ is the same as the asymptotic distribution of $coker(\probU_{n,n+u})$. This is an example of universality.

\begin{theorem*}\cite{arxivtwo}
The limit (\ref{eqn:sawinwood}) continues to hold when $\probU_{n,n+u}$ is replaced by $\probM_{n,n+u}$.
\newline
\end{theorem*}
\begin{remark} For $R \cong \Z/p^N \Z$, this theorem appeared in the work of Maples and Wood,\cite{Maples1}\cite{WoodIntegral}, with stronger results proven by Nguyen and Wood\cite{WoodNguyen}. When $R$ is the quotient of a DVR, this is proven by Yan, under a slightly different assumption on the distribution of the entries \cite{Yan}. 
\end{remark}

\paragraph{The main result: universality for ergodic averages.} In this note, we prove another manifestation of universality; we show that (\ref{eqn:ergodicuniform}) also holds when $\probU_{n,n+u}$ is replaced by $\probM_{n,n+u}$.
\begin{theorem}
\label{thm:ergodiciid}
Let $\probM_{i,i+u}$ be the top left $i \times i+u$ corner of $\probM$, where $\probM$ is the infinite random matrix defined above. Then,
\begin{equation}
\label{eqn:ergodiciid}
\frac{1}{N} \sum_{i=1}^N \One_{(coker(\probM_{i,i+u})=A)}
\xrightarrow{N \rightarrow \infty}
\mu_{u}(A) \text{ a.s.} \end{equation}
\end{theorem}
We deduce this from \cite{markovchainsoverr} and \cite{arxivtwo}. Indeed, as mentioned previously, \cite{markovchainsoverr} implies that 
\[
\hdots\hspace{0.1in} ,\hspace{0.1in} coker(\probU_{i,i+u}) \hspace{0.1in} ,\hspace{0.1in} coker(\probU_{i+1,i+u+1}) \hspace{0.1in} , \hspace{0.1in}  \hdots
\]
is a Markov chain generated by a certain operator, $\Delta_u$, while \cite{arxivtwo} implies that the process
\begin{equation}
\label{eqn:process}
\hdots\hspace{0.1in} ,\hspace{0.1in} coker(\probM_{i,i+u}) \hspace{0.1in} ,\hspace{0.1in}  coker(\probM_{i+1,i+u+1}) \hspace{0.1in} , \hspace{0.1in}  \hdots
\end{equation}
is "approximately" a Markov chain generated by $\Delta_u$, in a certain quantitative sense. This allows us to deduce (\ref{eqn:ergodiciid}) from the ergodic theorem for Markov chains.

\subsection{Acknowledgements}

The author would like to thank Hoi Nguyen, Alexander Van Werde and Melanie Wood for helpful discussions related to the subject of this paper. The author would also like to thank Alexander Yu for his encouragement.
\newline
\newline
\noindent
AI was not used in the course of this project.

\section{The corner process is approximately a Markov chain}

We use the symbol "$*$" to denote independent uniformly random variables. Let $T$ denote the group of upper triangular matrices with $1$'s on the diagonal, and denote by $t$ the map
\begin{equation}
\label{eqn:definitiont}
t: \, Mat \, \rightarrow \, T \, \backslash Mat \, /T
\end{equation}
that takes a matrix to its orbit under the action of $T \times T$. Finally, denote by $d_{TV}$ the total variation distance.
The inequality of \cite[Corollary immediately preceding \S 2.3]{arxivtwo} implies that

\begin{equation}
\label{eqn:dtvtinequality}
d_{TV}
\left(
t
\begin{mat}{ccc|c}
&&& * \\
&\probM_{n,n+u}&& \vdots \\
&&& * \\ \hline
*& \hdots & * & *
\end{mat}
,
t
\begin{mat}{ccc} 
&& \\
&\probM_{n+1,n+u+1} & \\
&&
\end{mat}
\right) < O(\theta^n)
\end{equation}
where $\theta<1$\rev{$\theta$} is an explicit constant (defined in \cite[Theorem 2.2]{arxivtwo}), that depends only on the distribution of the entries of $\probM$ and on the cardinality of the residue field of $R$.

\begin{lemma}
\label{lem:totvar}
The total variation distance between 
\[
\Big(
\begin{array}{ccccc}
\hdots\,,&
coker(\probM_{i,i+u})\,,&
\hdots\,,&
coker(\probM_{n,n+u})\,,&
coker(\probM_{n+1,n+u+1})
\end{array}
\Big)
\]
and
\[
\Big(
\begin{array}{ccccc}
\hdots\,,&
coker(\probM_{i,i+u})\,,&
\hdots\,,&
coker(\probM_{n,n+u})\,,&
\Delta_u coker(\probM_{n,n+u})
\end{array}
\Big)
\]
is bounded above by $O(\theta^n)$.
\end{lemma}

\begin{proof}
This is an immediate consequence of the inequality (\ref{eqn:dtvtinequality}).
\end{proof}

\begin{definition}
Denote by $X_n$ the following process, indexed by $i \in \mathbb{N}$:
\[
\Big( \hspace{0.1in} \hdots \hspace{0.1in}, coker(\probM_{i,i+u}) \,, \hspace{0.1in} \hdots \hspace{0.1in} ,
coker(\probM_{n,n+u})\,,  
\]
\[
\Delta_u coker(\probM_{n,n+u})\,,
 \hspace{0.1in} \hdots \hspace{0.1in} ,
\Delta_u^{i-n} coker(\probM_{n,n+u}) \,, \hspace{0.1in}
\hdots \hspace{0.1in} \Big)
\]
$X_\infty$ denotes the process:
\[
\Big( \hspace{0.1in} \hdots \hspace{0.1in}, coker(\probM_{i,i+u}) \,, \hspace{0.1in} \hdots \hspace{0.1in}  \Big)
\]
\end{definition}
\begin{corollary} (of \autoref{lem:totvar})
\begin{equation}
\label{eqn:xn}
d_{TV}(X_n,X_{n+1}) < O(\theta^n)
\end{equation}
\end{corollary}

\begin{theorem}
\label{thm:xinf}
$X_n$ converges to $X_\infty$ in total variation. More precisely:
\begin{equation}
\label{eqn:xinf}
d_{TV}(X_n, X_\infty) < O(\theta^n)
\end{equation}
\end{theorem}

\begin{proof}
(\ref{eqn:xn}) implies that for any $k$,
\[
d_{TV}(X_n,X_{n+k}) < O(\theta^n)
\]
where the implicit constant is different from the one in (\ref{eqn:xn}). Hence, $X_n$ is a Cauchy sequence in the total variation topology.
\newline
\newline
Now, the space of probability measures, endowed with the total variation topology, is complete. Therefore, there exists $Y$ such that 
\begin{equation}
\label{eqn: cauchy}
d_{TV}(X_n,Y) \leq O(\theta^n)
\end{equation}
For any $N$, and for large enough $n$, the pushforward of the distribution of $X_n$ to the first $N$ coordinates must coincide with the pushforward of the distribution of $X_\infty$  to the first $N$ coordinates. The inequality (\ref{eqn: cauchy}) implies that, for any $N$, the pushforward of the distribution of $Y$ to the first $N$ coordinates must coincide with the pushforward of the distribution of $X_\infty$  to the first $N$ coordinates. As this is true for any $N$, $X_{\infty}$ must have the same distribution as $Y$.
\end{proof}

\subsection{Corollaries}

\begin{corollary} 
The asymptotic distribution of $coker(\probM_{i,i+u})$ is $\mu_{u}$.
\end{corollary}

\begin{proof}
In the limit $i \rightarrow \infty$, the distribution of $(X_n)_i$ converges to $\mu_u$ in total variation, because $\Delta_u$ generates a recurrent Markov chain. By \autoref{thm:xinf},
\[
d_{TV}\Big(
(X_n)_i,
(X_{\infty})_{i}
\Big) < O(\theta^n)
\]
We take the $limsup$ as $i \rightarrow \infty$ to get:
\[
\limsup_{i \rightarrow \infty} \, d_{TV}\Big(\, \mu_0 \,, \,
(X_{\infty})_{i} \,
\Big) < O(\theta^n)
\]
and then take the limit as $n \rightarrow  \infty$. This shows that the distribution of $(X_\infty)_i$ must also converge to $\mu_u$. This proves the corollary.
\end{proof}

\begin{remark}
The statement of the corollary was previously demonstrated in \cite{arxivtwo}, with an explicit convergence rate. The result is proven again here, in order to show that it can be deduced from a dynamic perspective.
\end{remark}
The next corollary is the main result of this note.

\begin{corollary}[\autoref{thm:ergodiciid}]
\begin{equation}
\label{eqn:ergodicaverages}
\frac{1}{N} \sum_{i=1}^N \One_{(coker(\probM_{i,i+u})=A)}
\xrightarrow{N \rightarrow \infty}
\mu_{u}(A) \text{ a.s.}
\end{equation}
\end{corollary}

\begin{proof}
By the ergodic theorem for Markov chains, for any $n$,
\[
\probP
\left(
\frac{1}{N} \sum_{i=1}^N \One_{((X_n)_i=A )}
\xrightarrow{N \rightarrow \infty}
\mu_{u}(A)
\right) =1
\]
Hence, by (\ref{eqn:xinf}),
\[
\probP
\left(
\frac{1}{N} \sum_{i=1}^N \One_{((X_\infty)_i=A )}
\xrightarrow{N \rightarrow \infty}
\mu_{u}(A)
\right) \geq 1 - O(\theta^n)
\]
Taking $n\rightarrow \infty$ proves (\ref{eqn:ergodicaverages}). 
\end{proof}

\begin{remark}
Although we have proven (\ref{eqn:ergodicaverages}) only for i.i.d. matrices over finite local rings, (\ref{eqn:ergodicaverages}) in fact implies the analogous statement for i.i.d. matrices over $\Z_p$.
\end{remark}

\bibliographystyle{alpha}
\bibliography{ThesisBibliographyPrivetPrivet}
\end{document}